\documentclass[11 pt]{amsart}
\usepackage{comment}
\usepackage{enumerate}
\usepackage{amssymb, amsmath}

\def\Z{{\mathbb Z}}

\def\SL{{\rm SL}}
\def\GL{{\rm GL}}

\def\P{{\mathbb P}}
\def\Disc{{\rm Disc}}

\def\disc{{\rm disc}}

\def\R{{\mathbb R}}

\def\F{{\mathbb F}}

\def\Q{{\mathbb Q}}

\title[Thue equations failing the integral Hasse principle]{A positive proportion of Thue equations \\ fail the integral Hasse principle}

 \author{Shabnam Akhtari and Manjul Bhargava}
\address{ Fenton Hall\\
University of Oregon\\
Eugene, OR 97403-1222 USA}

 \email {akhtari@uoregon.edu}
 
 \address{\noindent Fine Hall, Washington Road\\
Princeton NJ 08544-1000 USA}

 \email{bhargava@math.princeton.edu}
 
\subjclass[2010]{}

\begin{document}

 \newtheorem{thm}{Theorem}[section]
\newtheorem{prop}[thm]{Proposition}
\newtheorem{lemma}[thm]{Lemma}
\newtheorem{cor}[thm]{Corollary}
\newtheorem{conj}[thm]{Conjecture} 

\begin{abstract}
For any nonzero $h\in\Z$, we prove that a positive proportion of integral binary cubic forms $F$ do locally everywhere represent $h$ but do not  globally represent $h$; that is, a positive proportion of cubic Thue equations $F(x,y)=h$ fail the integral Hasse principle.  Here, we order all classes of such integral binary cubic forms $F$ by their absolute discriminants. We prove the same result for Thue equations $G(x,y)=h$ of any fixed degree $n \geq 3$, provided that these integral binary $n$-ic forms $G$ are ordered by the maximum of the absolute values of their coefficients.
\end{abstract}

\maketitle

\section{Introduction}\label{Intro}

Fix $h\in\Z$. The purpose of this article is to prove the existence of many cubic {\it Thue equations}
\begin{equation}
F(x,y)=h
\end{equation}
that have no solutions $(x,y)$ in the integers; here $F$ is a binary cubic form of nonzero discriminant 
with coefficients in the integers.

It is easy to construct such forms if we arrange a local obstruction. For example, any binary cubic form congruent to 
$xy(x+y)$  modulo $2$  will never represent an odd integer $h$.  Also, if  $F(x,y)=h$ has no solution, and $G$ is a {\it proper subform} of $F$, i.e., $G(x,y)=F(ax+by,cx+dy)$ for some integer matrix $A=\bigl(\begin{smallmatrix}a&b\\c&d\end{smallmatrix}\bigr)$ with $|\!\det A|>1$, then clearly $G(x,y)=h$ will also have no integer solutions.  

Our goal  in this paper is to show for every integer $n \geq 3$ that many (indeed, a positive proportion) of binary forms  of degree $n$ are not proper subforms, locally represent $h$ at every place, but globally fail to represent~$h$.

\begin{thm}\label{maincubic}
Let $h$ be any nonzero integer. When integral binary cubic forms $F(x , y) \in \mathbb{Z}[x , y]$ are ordered by absolute discriminant, a positive proportion of the $\GL_2(\Z)$-classes of these forms $F$ have the following  properties:
\begin{enumerate}[{\rm (i)}]
\item they locally everywhere represent $h$ $($i.e., $F(x , y) = h$ has a solution in~$\R^2$ and in~$\Z_p^2$ for all $p);$
\item  they globally do not represent $h$ $($i.e., $F(x , y) = h$ has no solution in~$\mathbb{Z}^2);$ and
\item  they are maximal forms $($i.e., $F$ is not a proper subform of any other form$).$
\end{enumerate}
\end{thm}

\noindent
In other words, a positive proportion of cubic Thue equations $F(x,y)=h$ fail the integral Hasse principle, when classes of (maximal or all) integral binary cubic forms $F$ are ordered by absolute discriminant. 

More precisely, let $N_h(3,X)$ denote the number of $\GL_2(\Z)$-classes of integral binary cubic forms having absolute discriminant less than $X$ that are maximal, locally represent $h$, but do not globally represent $h$; and let $N(3,X)$ denote the total number of $\GL_2(\Z)$-classes of integral binary cubic forms having absolute discriminant less than~$X$.  Then, for any nonzero integer $h$, we prove that 
\begin{equation}
\liminf_{X\to\infty} \frac{N_h(3,X)}{N(3,X)}>0.
\end{equation}

In fact, our method of proof provides a lower bound for the positive density of classes of binary cubic forms that locally represent but do not globally represent a given integer $h$, and indeed an explicit construction of this positive density of forms. 

We prove an analogous result for binary forms of general degree, provided that these forms are ordered by the maximum of the absolute values of their coefficients.

\begin{thm}\label{maingeneral}
Let $h$ be any nonzero integer. When binary forms  $F(x , y) \in \mathbb{Z}[x , y]$ of a given degree $n \geq 3$ are ordered by the maximum of the absolute values of their coefficients, a positive proportion of them
have the following  properties:
\begin{enumerate}[{\rm (i)}]
\item they locally everywhere represent $h$ $($i.e., $F(x , y) = h$ has a solution in~$\R^2$ and in~$\Z_p^2$ for all $p);$
\item  they globally do not represent $h$ $($i.e., $F(x , y) = h$ has no solution in~$\mathbb{Z}^2);$ and 
\item  they are maximal forms $($i.e., $F$ is not a proper subform of any other form$).$
\end{enumerate}
\end{thm}

\noindent
In other words, a positive proportion of Thue equations $F(x,y)=h$ of degree $n$ fail the integral Hasse principle, when (maximal or all) integral binary $n$-ic forms~$F$ are ordered by the maximum of the absolute values of their coefficients.

More precisely, let $N^\ast_h(n,X)$ denote the number of integral binary $n$-ic forms whose coefficients are each less than $X$ in absolute value that are maximal, locally represent $h$, but do not globally represent $h$; and let $N^\ast(n,X)$ denote the total number of integral binary $n$-ic forms whose coefficients are each less than $X$ in absolute value.  Then, for any nonzero integer $h$, we prove that 
\begin{equation}
\liminf_{X\to\infty} \frac{N^\ast_h(n,X)}{N^\ast(n,X)}>0.
\end{equation}

As with Theorem~\ref{maincubic}, our method of proof provides an explicit lower bound for the positive density of binary $n$-ic forms that do not represent a  given integer, and moreover, yields an explicit construction of this positive density of forms.  

It is a folk conjecture that, for any degree $n\geq 3$, a density of  
100\% of integral binary $n$-ic forms that locally represent a fixed integer $h$ do not globally represent $h$
(either when classes of binary $n$-ic forms are ordered by absolute discriminant as in Theorem~\ref{maincubic}, or when binary $n$-ic forms are ordered by coefficient height as in Theorem~\ref{maingeneral}).  

We remark that for every fixed  even positive integer  $n$, one can deduce that a positive proportion of binary forms of degree $n$ do not represent a given integer $h$ by using the results in
\cite{hyperE}, where it is shown that most hyperelliptic curves $z^2=F(x,y)$ over $\mathbb{Q}$ (with any finite set of congruence conditions on the coefficients of the form $F(x,y)$) have no rational points. Indeed, when $n$ is even, the integer solutions to $F(x,y)=h$ may be viewed naturally as a subset of the rational points on the hyperelliptic curve defined by the equation $hz^2=F(x,y)$
in the weighted projective space $\mathbb P(1,1,n/2)$.  However, we note that the methods of \cite{hyperE} do not provide an explicit construction to produce such forms $F$.  

We also note that the techniques of \cite{hyperE} also currently apply only in the case that the degree~$n$ of $F$ is even.  
In particular, the results in, e.g., \cite{BG} and \cite{PS2} on ``odd degree'' hyperelliptic curves $z^2=x^n+a_1 x^{n-1}+\cdots+a_n$ do not apply to odd degree Thue equations $F(x,y)=h$ because, when the degree~$n$ of $F$ is odd, the integral solutions to $F(x,y)=h$ do not naturally map to rational points on the odd degree hyperelliptic curve $hz^2=F(x,1)$, unless $y=1$. 

We note that various interesting {\it parametric families} of Thue equations have been studied in the past.  The first study of a parametrized family of Thue equations in fact dates back to Thue himself~\cite{Thue}, and more recently occurs in the work of Thomas~\cite{Tho}.  We refer the reader to \cite{Heu} for a complete list of references up to the year 2005. Our goal in this article is of course to study the {\it universal family} of all Thue equations in each degree.  There do not appear to have been any general results on the universal family of Thue equations of a given degree in the literature previously, with the exception of the results in the case of even degree in~\cite{hyperE} described above. 

In~\cite{DM}, Dietmann and Marmon considered special diagonal Thue equations $ax^n+by^n=1$ ($n\geq 3$) and, under the ABC Conjecture, showed that most such Thue equations that are locally soluble fail the integral Hasse principle.  There have also been a number of recent works on families of non-Thue equations failing the integral Hasse principle, including that of Mitankin~\cite{Mit} on the density of failures of the integral Hasse principle among affine quadric surfaces, and of Jahnel and Schindler~\cite{JS} on examples of such failures among certain quartic del Pezzo surfaces. 
Many authors have considered such failures of the integral Hasse principle in terms of the integral Brauer--Manin obstruction, as introduced by Colliot-Th\'el\`ene and Xu~\cite{CTX}.  Theorems~\ref{maincubic} and~\ref{maingeneral} give a very classical family of affine equations (namely, Thue equations)---and to our knowledge the first known algebraic family of such equations of degree~$n$ at least~3---where a positive proportion of the equations are proven unconditionally to fail the integral Hasse principle. 

Our strategy to prove Theorems~\ref{maincubic} and \ref{maingeneral} is as follows.  
We make critical use of two important results on individual Thue equations.  The first result is due to Bombieri and Schmidt~\cite{Bom}, which in the form we shall need it, states that, given a squarefree integer $m$ relatively prime to~$h$ such that $F(x,y)$ factors completely into linear factors modulo $m$, the integer solutions of $F(x,y)=hm$ can be put in bijection with the integer solutions of $n^{\omega(m)}$ auxiliary Thue equations $G_j(x,y)=h$. (Here, $\omega(m)$ denotes the number of prime factors of $m$.)  The second result, proven by the first author~\cite[Theorem~1.2]{AkhQuaterly}, extending earlier works of Evertse--Gy\H{o}ry~\cite[Theorem 3]{EG} and Stewart~\cite[Theorem~1]{Ste},  states that a Thue inequality $|F(x,y)|\leq m$ of degree $n$ has at most $B$ solutions, where $B$ is a linear function of $n$, provided that the discriminant of $F$ is sufficiently large compared to $m$.  

Our method then is to play these results off of each other: if $n^{\omega(m)}$ is larger than $B$, and the discriminant of $F$ is sufficiently large relative to $hm$, so that $F(x,y)=hm$ has at most $B$ integer solutions, then most of the $n^{\omega(m)}$ auxiliary Thue equations $G_j(x,y)=h$ cannot have any integer solutions!  

If we can ensure that the Thue equations $G_j(x,y)=h$ are locally soluble at all places, then we will have produced many Thue equations $G_j(x,y)=h$ that are locally soluble but globally insoluble.  Thus, given $h\in\Z$, we suitably construct $F$'s and $m$'s to produce enough equations $G_j(x,y)=h$ that are locally soluble but globally insoluble, so that in fact the $G_j$'s form a positive proportion of all integral binary forms of degree $n$. 

Our method should likely apply to prove the analogues of Theorems~\ref{maincubic} and \ref{maingeneral}  over number fields, provided the analogues of the results of \cite{EG}, \cite{Ste}, or \cite{AkhQuaterly} used above could also be extended over number fields; the necessary techniques seem to exist in the literature, but as far as we know, this analogue over number fields has not yet been worked out.  (The constructions in this paper, and those in the work of Bombieri and Schmidt~\cite{Bom}, do hold without essential change over any number field.)

This paper is organized as follows.  In Sections \ref{EF} and \ref{PC}, we discuss the necessary preliminaries and essential results we need from \cite{AkhQuaterly} and \cite{Bom}, respectively.  We then turn to the proof of Theorem~\ref{maincubic} in Section~\ref{proofofmain}.  This consists of: (i) construction of the forms $F$ and integers $m$ when $h=1$; (ii) construction of the forms $G_j$; (iii) proof that many of the $G_j$'s do not represent 1; (iv) proof that many of the $G_j$'s do represent $1$ everywhere locally; (v) proof that the $G_j$'s yield a positive proportion of all integer binary cubic forms, up to equivalence, when ordered by absolute discriminant; and (iv) extension of the arguments to general $h$.  Finally, in Section~\ref{proofofmaingeneral}, we carry out the analogues of the arguments in Section~4 for Thue equations of general degree $n\geq 3$ ordered by height, thus proving Theorem~\ref{maingeneral}.

\section{Preliminaries}\label{EF}
 
Let $F(x , y) \in \mathbb{Z}[x , y]$ be an integral binary form and $
A = \bigl( \begin{smallmatrix}
a & b \\
c & d \end{smallmatrix} \bigr) \in \textrm{GL}_{2}(\mathbb{Q})$. Then define the binary form $F^{A}$ by
$$
F^{A}(x , y) : = F(ax + by ,\  cx + dy).
$$
Then, for $A,B\in \GL_2(\Q)$, we have $(F^A)^B(x,y)=F^{AB}(x,y)$.  Furthermore, if $(x_{0} , y_{0})$ is a   solution of $F(x , y) =h$, then 
$$A \left( \begin{array}{c}
x_{0}\\
y_{0}
\end{array} \right)= \left( \begin{array}{c}
ax_{0}+by_{0}\\
cx_{0}+dy_{0}
\end{array} \right),$$
and so  $(ax_{0}+by_{0} , cx_{0}+dy_{0})$ is a  solution of $F^{A^{-1}}(x , y) = h$.

If $F(x , y)$  factors in $\mathbb{C}$ as
$$
F(x , y) = \prod_{i=1}^{n} (\alpha_{i} x - \beta_{i}y), 
$$
then the discriminant $\textrm{disc}(F)$ of $F$ is given by
$$
\textrm{disc}(F) = \prod_{i<j} (\alpha_{i} \beta_{j} - \alpha_{j}\beta_{i})^2.
$$
It follows that, for any $2 \times 2$ matrix $A$, we have
\begin{equation}\label{St6}
\textrm{disc}(F^{A}) = (\textrm{det} A)^{n (n-1)} \textrm{disc}(F).
\end{equation}
If $A \in \GL_{2}(\mathbb{Z})$,  then we say that $F^{A}$ is {\it equivalent} to $F$.

A pair $(x_{0} , y_{0}) \in \mathbb{Z}^2$ is called a \emph{primitive solution} to the  equation $F(x , y) = h$ if  $F(x_{0} , y_{0}) = h$ and  
$\gcd(x_{0} , y_{0}) = 1$.
 We denote by $N_{F, h}$ the number of primitive solutions in integers $x$ and $y$ of the  equation $F(x , y) = h$. If $F_{1}$ and $F_{2}$ are equivalent,  then evidently $\textrm{disc}(F_{1})= \textrm{disc}(F_{2})$ and 
 $N_{F_{1}, h} = N_{F_{2}, h}$.

The following is an immediate consequence of the main results in \cite{AkhQuaterly}.
 
 \begin{thm}\label{maineq}
 Let $F(x , y) \in \mathbb{Z}[x , y]$ be an irreducible binary form of degree $n\geq 3$ and discriminant $D$.  Let $m$ be an integer with
  $$
0 < m  \leq \frac{|D|^{\frac{1}{2(n-1)} - \epsilon} }  {(3.5)^{n/2} n^{ \frac{n}{2(n-1) } } },
 $$
 where $ 0< \epsilon < \frac{1}{2(n-1)}$.
 Then the equation $F(x , y) = m$ has at most
  \[ \left\{
  \begin{array}{l l}
  7n + \frac{n}{(n-1) \epsilon}  & \quad \text{if}\, \,  n \geq 5\\
  9n + \frac{n}{(n-1) \epsilon}  & \quad \text{if} \, \, n = 3, 4 
  \end{array} 
 \right.  \]
   primitive solutions.  $($If $n$ is even, then two solutions $(x_{0}, y_{0})$ and $(-x_{0}, -y_{0})$ are considered here to be one solution$)$. 
   \end{thm}
 We will use Theorem~\ref{maineq}, together with some results on polynomial congruences recalled in the next section, in our construction of binary forms that do not represent a given integer $h$.

\section{Polynomial congruences} \label{PC}

In this section, we describe two lemmas on the zeros of polynomials modulo prime powers, due to Bombieri and Schmidt \cite[Lemma~7]{Bom} and Stewart \cite[Theorem~2]{Ste}, which will also form key ingredients in our construction of forms that do not represent  a given integer $h$.  

To state these results, let $\Omega_{p}$ be an algebraic closure of the $p$-adic field~$\mathbb{Q}_{p}$. We denote by $|\cdot |_{p}$ the usual absolute value in $\mathbb{Q}_{p}$, together with an extension to $\Omega_{p}$. Here $\mathbb{Z}_{p}$ is the ring of integers of $\mathbb{Q}_{p}$ and $R_{p}$ denotes the ring of elements $\zeta$ in $\Omega_{p}$ for which $|\zeta|_{p}\leq 1$.  Then we have the following lemma  due to Stewart in \cite[Theorem~2]{Ste}.
\begin{lemma}\label{L7Bom}
Let $f(X)$ be a polynomial of degree $n$  with integer coefficients, content $1$ and nonzero discriminant $D$. Let $k$ be a positive integer,  $p$ be a prime number, and $s$ be the number of zeros of $f$ in $R_{p}$. Then the complete solution of the congruence 
$$
f(X) \equiv 0 \, \, ({\rm{mod}}\, \, p^k)
$$
is given by at most $s$ congruences
$$
X \equiv c_{i}\, \, ({\rm{mod}}\, \,  p^{k-u_{i}}),
$$
where $c_{i}$ are non-negative integers and $0 \leq u_{i} \leq k$. In addition, the above congruences can be found so that when $u_{i} \neq 0$, for each root $\alpha$ of $f(X)$, we have
\begin{equation}\label{Aim}
\left| \alpha - c_{i} \right|_{p} \geq p^{-k+u_{i}}.
\end{equation}
\end{lemma}
Now suppose  $f(X)$ is a polynomial of degree $n$, content $1$, and nonzero discriminant. Let $p$ be a prime and $k$ a non-negative integer. By Lemma~\ref{L7Bom}, we can assume that the complete solution of the congruence 
$$
f(X) \equiv 0 \, \, (\textrm{mod}\, p^k)
$$
is given by $t$ congruences
$$
X \equiv c_{i}\, \, (\textrm{mod}\, p^{k-u_{i}}),
$$
where $c_{i}$ are non-negative integers and $0 \leq u_{i} \leq k$. Here $t$ does not exceed the number of zeros of $f$ in $R_{p}$ and we can order the above congruences so that $u_{i} > 0$ for $0 < i\leq t_{1}$ and $u_{i} = 0$ for $t_{1}< i \leq t$. 

Set
\begin{equation*}
f_{i}(X): = f(p^{k-u_{i}} X + c_{i}). 
\end{equation*}
Then we have the following lemma, which is a part of Theorem 2 of \cite{Ste}.
\begin{lemma}\label{content}
$p^{k}$ divides the content of 
$f_{i}(X)$,
for $i = 1, \ldots t$. 
\end{lemma}

Thus we can define  $t$ (with  $ t \leq s$) polynomials
$$
\tilde{f}_{i}(X) = p^{-k} f(p^{k-u_{i}} X + c_{i})
$$
with integer coefficients.

\section{Proof of Theorem \ref{maincubic} }\label{proofofmain}

Our strategy to prove Theorem \ref{maincubic} is as follows.  We first restrict to the case of $h=1$, and construct in Subsection~\ref{fconstruction} a set of maximal binary cubic forms $F(x,y)$ of positive density defined by congruence conditions.  For any such irreducible binary cubic form $F(x,y)$ of sufficiently large discriminant, we show in Subsection~\ref{gjconstruction} that each solution of $F(x,y)= m$, for an integer~$m$, corresponds uniquely to a solution of $G_j(x,y)=1$, where $G_j$ lies in a certain set of 81 maximal binary cubic forms $G_j(x,y)=1$ associated to~$F$ with $\Disc(G_j)\ll \Disc(F)$.  These $G_j$'s are constructed using the polynomial congruence results of Section~\ref{PC}.  Using the fact that $F$ can represent $m$ only an absolutely bounded number of times by Theorem~\ref{maineq} when $F$ has sufficiently large discriminant, we conclude in Subsection~\ref{gjno1} that most of the $G_j$'s cannot represent 1. 

In Subsection~\ref{locsol}, we argue that these forms $G_j$ that globally do not represent 1 do represent 1 everywhere locally.  In Subsection~\ref{posdens}, we show that the $G_j$'s yield a positive density of all binary cubic forms when ordered by absolute discriminant.  This proves Theorem~\ref{maincubic} in the case $h=1$.  
We explain how the proof extends naturally to general $h$ in Subsection~\ref{genh}.

We start by assuming $h=1$. 

\subsection{Construction of binary cubic forms $F(x,y)$ satisfying certain congruence conditions and having large discriminant}\label{fconstruction}

Let $p_1=5$, $p_2=7$, and in general $p_{i}$ the $i$-th prime greater than 3, and set
$$m = \prod_{i=1}^{k} p_{i}, $$
 where $k\geq 4$ is any integer. 
 Let $F(x , y)$ be a maximal primitive irreducible integral binary cubic form having splitting field with Galois group $S_3$ such that
 $$
 |\textrm{disc}(F)| > m^{20} (3.5)^{30} 3^{15}.
 $$
Assume further that $F(x,y)$ does not split completely modulo~2, does split completely modulo the primes $p_{1}$, \ldots, $p_{k}$, and is irreducible modulo~3. (We say that $F(x,y)$ {\it splits completely} {modulo} $p$ if it can be expressed modulo~$p$ as a product of pairwise non-proportional linear forms with coefficients in $\Z/p\Z$.) Finally, assume, for each prime $p>p_k$ with $p\equiv 1$ (mod $3$), that $F(x,y)$ does not factor as $cL(x,y)^3$ modulo~$p$ for any linear form $L$ and constant $c$ over $\F_p$. 
We will show in Subsection~\ref{posdens} that such forms $F$ 
give a positive proportion of all classes of binary cubic forms when ordered by absolute discriminant.

\subsection{Construction of forms $G_j$ associated to each $F$}\label{gjconstruction}

Given any form $F$ as constructed in Subsection~\ref{fconstruction}, we will show how to obtain many mutually inequivalent irreducible maximal binary cubic forms $G$ that locally everywhere represent~1 but do not globally represent~1.

First, because of the condition on the discriminant of $F$, 
we have by Theorem \ref{maineq} that the equation  
  $$
  F(x , y)= m
  $$ 
  has at most $34$  primitive solutions (take $\epsilon = \frac{1}{5}$).

Assume that
\begin{equation}\label{abcsplit}
F(x , y) \equiv m_{0} (x - a_{1}y) (x - b_{1}y) (x - c_{1}y) \, \, (\textrm{mod}\, \,  p_{1}),
\end{equation}
where  $m_{0} \not \equiv 0$ (mod $p_{1}$), and  $a_{1}, b_{1}$ and $c_{1}$ are three distinct integers modulo~$p_1$.
Following the method of Bombieri and Schmidt in~\cite{Bom}, define
\begin{eqnarray*}
F_{a_1}(x , y) &: =  & F(p_{1} x + a_{1}y , y),\\
F_{b_1}(x , y) &: =  & F(p_{1} x + b_{1}y , y),\\
F_{c_1}(x , y) & :=  & F(p_{1} x + c_{1}y , y).
\end{eqnarray*}
We claim that these three forms are pairwise inequivalent.  Indeed, any transformation $B\in\GL_2(\Q)$ taking $F_{a_1}(x,y)$ to $F_{b_1}(x,y)$ must be of the form $B=\bigl(\begin{smallmatrix}p_1&a_1\\ 0& 1\end{smallmatrix}\bigr)^{-1}\!A\bigl(\begin{smallmatrix}p_1&b_1\\ 0& 1\end{smallmatrix}\bigr)$, where $A\in\GL_2(\Q)$ stabilizes $F(x,y)$.  However, since $F(x,y)$ is irreducible and has Galois group $S_3$, the transformation $A$ must be the identity, and so $B=\bigl(\begin{smallmatrix}p_1&a_1\\ 0& 1\end{smallmatrix}\bigr)^{-1}\bigl(\begin{smallmatrix}p_1&b_1\\ 0& 1\end{smallmatrix}\bigr)$.  Thus $B\notin\GL_2(\Z)$, since $p_1\nmid (a_1-b_1)$, and so $F_{a_1}(x,y)$ and $F_{b_1}(x,y)$ are not $\GL_2(\Z)$-equivalent. 

By Lemma \ref{content}, or via direct substitution, we see that the contents of $F_{a_1}(x , y)$, $F_{b_1}(x , y)$ and $F_{c_1}(x , y)$ are exactly  divisible by $p_{1}$.  
Thus we may divide all the coefficients of $F_{a_1}$,  $F_{b_1}$ and  $F_{c_1}$ by $p_{1}$ to get three pairwise inequivalent binary cubic forms $\tilde{F}_{a_1}$,  $\tilde{F}_{b_1}$ and  $\tilde{F}_{c_1}$, with integral coefficients, content~$1$, and discriminant equal to~$p_{1}^2 \, \textrm{disc}(F)$.  

Since we assumed $F(x , y)$ splits completely modulo $p_{1}$, we either have \eqref{abcsplit} or
$$
F(x , y) \equiv m_{0} y (x - b_{1}y) (x - c_{1}y) \, \, (\textrm{mod}\, \,  p_{1}),
$$
for some integers  $m_{0} \not \equiv 0$  (mod $p_{1}$), $b_{1}$ and $c_{1}$, with $b_{1} \not \equiv c_{1}$ (mod $p_{1}$). In the latter case, we define $F_{b_{1}}$ and $F_{c_{1}}$ as above and set
$$
F_{\infty}(x , y) := F (p_{1}y , x)
$$
and 
$$
\tilde{F}_{\infty}(x , y) := \frac{1}{p_1}F (p_{1}y , x).
$$
Again, the unique element in $\GL_2(\Q)$ that transforms $F_{\infty}(x , y)$ to $F_{b_1}(x , y)$ is  $\bigl(\begin{smallmatrix}0&p_1\\ 1& 0\end{smallmatrix}\bigr)^{-1}\bigl(\begin{smallmatrix}p_1&b_1\\ 0& 1\end{smallmatrix}\bigr)$, and  this change-of-variable matrix does not belong to $\GL_{2}(\mathbb{Z})$.  Therefore,  $F_{\infty}(x , y)$, $F_{b_1}(x , y)$ and $F_{c_1}(x , y)$ are pairwise inequivalent.

By Lemma \ref{L7Bom},
 each  solution of  the equation $F(x , y)= m$ corresponds uniquely to a solution of one of 
 \begin{eqnarray*}
 \tilde{F}_{a_1} (x , y) &=& m/p_{1} ,\\
   \tilde{F}_{b_1} (x , y)& = & m/p_{1},\\
   \tilde{F}_{c_1} (x , y)  &=&  m/p_{1}.
\end{eqnarray*}
(Here we take $a_{1} = \infty$ if $F(x , y)$ is a multiple of $y$ modulo $p_{1}$.)

 Applying the above procedure, with the prime $p_2$ in place of $p_1$, to each of $ \tilde{F}_{a_1} (x , y)$,  $\tilde{F}_{b_1} (x , y)$ and  $\tilde{F}_{c_1} (x , y)$ (note that these forms still split completely modulo $p_2,\ldots,p_k$), we then obtain  $9$ different  cubic Thue equations of the shape
 $$
 K(x , y) = \frac{m}{p_{1} p_{2}},
 $$
where $K(x , y) \in \mathbb{Z}[x , y]$ has discriminant $p_1^2p_2^2\disc(F)$.

Applying the above procedure for each prime $p_{i}$, $i = 1, \ldots, k$, we then obtain 
$3^{k}$ binary cubic forms with integral coefficients, content $1$, and discriminant 
$$
\left(\prod_{i=1}^{k} p_{i}^2\right) \, \textrm{disc}(F).
$$
 We denote these binary forms $G_{j}$, for $j = 1, \ldots, 3^k$. 

Note that if $F(x , y)$ is  irreducible over $\mathbb{Q}$, then  its associated forms $G_{j} (x , y)$ will also  be irreducible over $\mathbb{Q}$.  Furthermore, the $G_{j}$'s are not constructed as proper subforms of the   binary form cubic $F(x , y)$, and in fact the $G_j$'s must be maximal forms as well.  Indeed, they are maximal over~$\Z_p$ for all $p\notin \{p_1,\ldots,p_k\}$ (being equivalent, up to a unit constant, to $F(x,y)$ over~$\Z_p$ in that case), while for $p\in\{p_1,\ldots,p_k\}$, we have $p\nmid \disc(F)$, implying $p^2||\disc(G_j)$, and so $G_j$ cannot be a subform over $\Z_p$ of any form by equation~(\ref{St6}). Hence the $G_j$'s are all maximal forms.  

 \subsection{Many of the $G_j$'s do not represent 1}\label{gjno1}

Let $F$ be as in Subsection~\ref{fconstruction}, and let the forms $G_j$, $j=1, \ldots, 3^k$, be constructed as in Subsection~\ref{gjconstruction}. 
 Each primitive solution $(x_{0} , y_{0})$ to the cubic equation 
 $$
 F(x , y) = m = \prod_{i=1}^{k} p_{i}
$$
 corresponds to a unique triple $(j, x_j, y_j)$ so that $(x_j, y_j)$ is a solution to the equation 
 $$
 G_{j}(x , y) = 1.
 $$
 Since the equation $F(x , y) = m$ has at most $34$ primitive solutions, at least $3^{k} -  34$ binary forms among the $G_{j}$'s cannot represent $1$.
 
Thus starting from any maximal irreducible binary cubic form $F(x,y)$ that splits completely modulo $p_1,p_2,\ldots,p_k$, we have produced at least $3^k-34$ maximal irreducible binary cubic forms $G(x,y)$ that do not represent~1. 

 Note that $k= 4$ suffices to produce a positive number of such forms $G(x,y)$: if $k=4$, then from such an $F(x,y)$ above we obtain $3^4 - 34 = 47$ binary cubic forms $G(x , y)$ such that $G(x , y) = 1$ does not have any solution in integers $x$ and $y$. 
 
 \subsection{The cubic forms $F$ and also the forms $G_j$ represent 1 everywhere locally}\label{locsol}

 Next we wish to show that these binary forms $G_{j}$ that do not globally represent 1 do represent $1$ everywhere locally.  
 
 Since any binary cubic form takes positive values, we conclude that each~$G_j$ certainly takes the value 1 over $\R$.  To see that the $G_j$'s represent~1 over~$\Z_p$ for all primes~$p$, 
  we use the following lemma. 
 
 \begin{lemma}\label{localrep1}
 A primitive binary cubic form $f$ over $\Z_2$ represents $1$ if and only if $f(x,y)\not\equiv xy(x+y)\pmod2$.
 If $p\equiv 1\pmod6$, then a primitive binary cubic form $f$ over $\Z_p$ represents $1$ if and only if 
 $f(x,y)\not\equiv c L(x,y)^3\pmod p$, where $L$ is a linear form and $c$ is a non-cube modulo $p$.  Finally, 
 if $p\equiv 5\pmod 6$, then every primitive binary cubic form $f$ over $\Z_p$ represents~$1$. 
 \end{lemma}
 
 \begin{proof}
 If $p\equiv 2$ (mod 3), then every unit in $\Z_p$ is a cube. Thus, in this case, every binary form $f$ over $\Z_p$ that represents a unit represents 1: indeed, if $f(x,y)=u$ for $(x,y)\in\Z_p^2$ and $u\in\Z_p^\times$, then $f(x/a,y/a)=1$, where $a^3=u$.  Now every primitive binary cubic form over $\Z_p$ represents a unit in~$\Z_p$, except when $p=2$ and  $f(x,y)\equiv xy(x+y)$ (mod~2).

It remains to treat the case $p\equiv 1$ (mod 3).  Let $c$ be a cubic non-residue modulo $p$. Then if $f(x,y)\equiv c (x - by)^3$ (mod~$p$), then clearly $f$ does not represent $1$ modulo $p$ (for any $b \in \mathbb{F}_p$). We show this is the only case where a local obstruction over $\Z_p$ will arise. 

Indeed, if $c$ is a cubic residue modulo~$p$, and $f(x,y)\equiv c(x - by)^3$ (mod~$p$) for some $b$, then $f$ evidently represents unit cubes, and therefore~1, over $\Z_p$. 

Next let $d$, $a$, and $b$ be three integers, with  $ d\not \equiv 0$   and  $a \not \equiv b$ (mod~$p$). Then
 the binary form 
 $$
 d (x -ay) (x - by)^2
 $$
 represents $1$ modulo $p$. This is easy to verify by solving the system of linear equations 
 \begin{eqnarray*}
 x - a y &= 1\\
 x - by &= d.
 \end{eqnarray*}
 Thus if a binary form $f(x,y)$ over $\Z_p$ has a double root but not a triple root modulo $p$, then it represents a unit cube and therefore 1.
 
 Finally, if a binary form $f(x , y)$ over $\Z_p$ does not have any multiple root modulo~$p$, then the curve
 $$
 z^3 = f(x , y) \pmod{p}
 $$
 over $\F_p$ is smooth of genus $1$; therefore by the Hasse bound, the above equation has solutions with $z \not \equiv 0$ (mod $p$). Thus $f(x , y)$ represents~$1$ modulo~$p$, so represents unit cubes and thus~1 over $\Z_p$. 
 \end{proof}
 
By the assumptions on how $F$ factors modulo each prime, and using the fact that each $G_j$ factors in a predictable way in relation to $F$ modulo every prime, we can then see that each $G_j$ locally represents 1.  Indeed, $F$ and thus each $G_j$ is not congruent to $xy(x+y)$ (mod 2); each $G_j(x,y)$ factors as $c\,  L(x,y)y^2$ for some constant $c$ modulo each prime $p\in\{p_1,\ldots,p_k\}$; and each $G_j$ does not factor modulo $p$ as a constant times the cube of a linear form for each prime $p > p_k$ for which $p\equiv 1$ (mod $3$).  By Lemma~\ref{localrep1}, we conclude that $F$ and the $G_j$'s all represent 1 over $\Z_p$ for all $p\neq 3$.  

For $p=3$, we note that the equation $z^3=G_j(x,y)$ has a solution over $\Z_3$ where $z$ is a unit, and thus $G_j(x,y)$ represents~1 over $\Z_3$ as well.  Finally, any nonzero binary cubic form over $\R$ represents positive values and thus represents~1 over $\R$; thus the forms $G_j$ also all locally represent 1 over $\R$.

In summary, we have the following theorem.

\begin{thm}\label{mainconstruction}
Let $p_{i}$ be the $i$-th odd prime greater than $3$ and set 
$m = \prod_{i=1}^{k} p_{i}$,
 where $k\geq 4$ is any integer. 
 Suppose $F(x , y)$ is a maximal primitive irreducible integral binary cubic form, with Galois group $S_3$ and discriminant greater than $m^{20} (3.5)^{30} 3^{15}$, such that $F(x,y)$ is irreducible modulo~$3$, does not split completely modulo $2$, but does split completely modulo the primes $p_{1},\ldots,p_{k}$.  Further assume for each prime $p>p_k$ with $p\equiv 1\pmod3$ that $F(x,y)$ does not factor as $cL(x,y)^3$ modulo~$p$ for any linear form $L$ and constant $c$ over $\F_p$. 
Given any such form $F$, there exist $3^k$ mutually inequivalent irreducible maximal binary cubic forms $G_{j}$, for $j =1, \ldots, 3^k$, having discriminant $\disc(F)m^2$ such that  at least $3^k-34$ of the forms~$G_j$ satisfy the following:
\begin{enumerate}[{\rm (i)}]
\item they locally everywhere represent $1;$
\item  they globally do not represent $1;$ and
\item  they are maximal forms.
\end{enumerate}\end{thm}

 \subsection{The forms $F$ and also the forms $G_j$ yield a positive proportion of all classes of forms}\label{posdens}
 
We now verify that the classes of the forms $F$ satisfying the conditions of Theorem~\ref{mainconstruction} yield a positive proportion of all classes of integral binary cubic forms when ordered by absolute discriminant.  For this, we use the following slight refinement of a theorem of Davenport and Heilbronn:

\begin{thm}[Davenport--Heilbronn~\cite{DH}]\label{dav} 
Let $S$ be a set of integral binary cubic forms $F$ defined by congruence conditions modulo bounded powers of primes $p$ where, for sufficiently large $p$, the defining congruence conditions at $p$ exclude only a set of forms $F$ satisfying $p^2\mid \Disc(F)$. 
Let $h(S;D)$ denote the number of classes of irreducible binary cubic forms with Galois group $S_3$ that are contained in $S$ and have discriminant $D$.  Then:
\begin{enumerate}[{\rm (a)}]
\item $\displaystyle{\sum_{-X <D < 0} h(S;D) \sim \frac{\pi^2}{24} X \prod_p\mu_p(S);}$
\item  $\displaystyle{\;\:\sum_{0 <D < X} \,h(S;D) \sim \frac{\pi^2}{72} X\prod_p\mu_p(S),}$
\end{enumerate}
where $\mu_p(S)$ denotes the $p$-adic density of $S$ in the space of integral binary cubic forms. 
\end{thm}
Davenport and Heilbronn's original theorem does not include the condition ``with Galois group $S_3$''; however, this modification is immediate in case~(a) of the theorem, since an irreducible integral binary cubic form with only one real root in $\P^1$ automatically has Galois group $S_3$; meanwhile, in case (b), the result follows from \cite[Theorem~3]{BY}, which is a quantitative version of Hilbert's Irreducibility Theorem for binary cubic forms, when (classes of) binary cubic forms are ordered by discriminant rather than by coefficient height. 

Theorem~\ref{dav} thus
 states that the number of equivalence classes of irreducible binary cubic forms per discriminant (which then have $S_3$ as their Galois group $100\%$ of the time) is a positive constant on average, even if we are counting only those forms satisfying a finite (or even a suitable infinite) set of congruence conditions.  It is a theorem of Shintani~\cite{shintani} that the same is true also for reducible forms.
 
In particular, Theorem~\ref{dav} immediately implies that the set $S_0$ of all~$F$ satisfying the conditions of Theorem~\ref{mainconstruction} gives a positive proportion of all classes of integral binary cubic forms, when these classes are ordered by absolute discriminant. Indeed, we may determine the exact density of such~$F$ among all irreducible binary cubic forms as follows. 

First, we note that if a form is non-maximal, then either it is not primitive, or after an $\SL_2(\Z)$-transformation is of the form $f(x,y)=a_0x^3+a_1x^2y+a_2xy^2+a_3y^3$, where $p^i\mid a_i$, $i=1,2,3$, for some prime $p$.  The proportion of nonzero binary forms modulo~$p$ having a triple root in $\F_p$ is $(p+1)(p-1)/p^4$; furthermore, if an integral binary form having a triple root modulo~$p$ is expressed, via an $\SL_2(\Z)$-transformation, in the form $f(x,y)=a_0x^3+a_1x^2y+a_2xy^2+a_3y^3$ where $p$ divides $a_1$, $a_2$, and $a_3$, then the probability that $f(x,y)$ additionally satisfies $p^i\mid a_i$ for $i=1,2,3$ is $1/p^3$. Therefore, for primes $p>p_k$ such that $p\equiv 2$ (mod~3), the $p$-adic density of binary cubic forms  that are primitive and maximal at $p$ is
given by 
$$\mu_p(S_0)=1-\frac1{p^4}-\frac{p^2-1}{p^7}.$$

The 2-adic density of all binary cubic forms that are primitive, maximal, and not congruent to $xy(x+y)$ (mod 2) is $\mu_2(S_0)=1-1/2^4-3/2^6-1/2^4=53/64$.  
By \cite[Lemma~17]{BST}, 
the $p_i$-adic density (for $i=1,\ldots,k$) of all primitive forms that split completely modulo $p_i$ is 
$$\mu_{p_i}=\frac{1}{6} \frac{(p_i-1)^2 p_i (p_i+1)} {p_i^4}; $$
the 3-adic density of all binary cubic forms that are primitive and irreducible is 
$$\mu_{3}=\frac{1}{3} \frac{(3-1)^2 3 (3+1)} {3^4} = {16}/{81}.$$

For primes $p>p_k$ such that $p\equiv 1$ (mod 3), the density of binary cubic forms  that are primitive and not congruent to 
$cL^3$ modulo $p$ for some linear form $L$ (and thus automatically maximal) and some constant $c$ over $\F_p$ is given by 
$$\mu_p(S_0)=1-\frac{(p-1)(p+1)}{p^4}.$$

Therefore, the density of $F$ satisfying the conditions of Theorem~\ref{mainconstruction} among all classes of irreducible integral binary cubic forms, when these forms are ordered by absolute discriminant, is given by
$$
\frac{53}{64}\cdot \frac{16}{81} \cdot\prod_{i=1}^{k} \frac{1}{6} 
\frac{p_i^4-p_i^3-p_i^2+p_i}{p_i^4}
\!\!\!\!\!\prod_{{p>p_k}\atop{p\equiv 1\!\!\!\!\!\pmod 3}} \!\!\!\!\!\!\left(1-\frac{p^2-1}{p^4}\right) \!\!\!\!\!
\prod_{{p>p_k}\atop{p\equiv 2\!\!\!\!\!\pmod 3}} \!\!\!\!\!\!\left(1-\frac{p^3+p^2-1}{p^7}\right).
$$

In order  to show that the forms $G_j$ that we have constructed from these forms $F$, which are maximal, locally represent 1, but do not globally represent 1, yield a positive proportion of all classes of irreducible integral binary cubic forms, we note that the $G_j$'s are mutually inequivalent, and that each $F$ gives rise to at least 47 $G_j$'s (when $k=4$), each satisfying $$\disc(G_j)= \left(\prod_{i=1}^k p_i^2 \right)\disc(F).$$
We conclude that the density of the classes of our forms $G_j$ among all classes of irreducible integral binary cubic forms is given by
$$
47 \cdot \frac{53}{64}\cdot \frac{16}{81} \cdot\prod_{i=1}^{k} \frac{1}{6} 
\frac{p_i^4-p_i^3-p_i^2+p_i}{p_i^4}\frac1{p_i^2}
\!\!\!\!\!\prod_{{p>p_k}\atop{p\equiv 1\!\!\!\!\!\pmod 3}} \!\!\!\!\!\!\left(1-\frac{p^2-1}{p^4}\right) \!\!\!\!\!
\!\prod_{{p>p_k}\atop{p\equiv 2\!\!\!\!\!\pmod 3}} \!\!\!\!\!\!\left(1-\frac{p^3+p^2-1}{p^7}\right)
$$
when $k=4$, and we have proven Theorem~\ref{maincubic} in the case $h=1$.

\subsection{The case of general $h$}\label{genh}

We may handle general $h\in\Z$ in a similar manner. 
Let $k$ again be a positive integer such that $3^k > 34$, 
let $p_{1}, \ldots, p_{k}$ be  the $k$ smallest primes greater than  $3$ such that $p_{i} \nmid h$, and 
set 
$$m = \prod_{i=1}^{k} p_{i}. 
$$
 Let $F(x , y)$ be a binary cubic form with $$
 |\textrm{disc}(F)| > (hm)^{20} (3.5)^{30} 3^{15}.
 $$
  Then by Theorem \ref{maineq}, the equation  
  \begin{equation}\label{Fhm}
  F(x , y)= h m
  \end{equation}
  has at most $34$  primitive solutions.

  In addition to the assumptions on the binary cubic form $F$ that were made in Subsection~\ref{fconstruction}, we also assume    that for every prime $p \mid h$, we have
  $$
F(x , y) \equiv L_1(x,y) L_2(x , y)^{2}\, \, \textrm{ (mod}\, \,  p),
$$
where $L_1(x,y)$ and $L_2(x,y)$ are linear forms over $\F_p$ that are linearly independent.  

  Similarly to the case $h=1$, performing the construction in Lemmas~\ref{L7Bom} and \ref{content} as in Subsection~\ref{gjconstruction} using the primes $p_1,\ldots,p_k$, 
  we obtain  $3^k$ mutually inequivalent cubic Thue equations of the shape
   \begin{equation}\label{Ghm}
  G_{j}(x , y) = h,
  \end{equation}
  so that every solution $(x_{0}, y_{0})$ to the equation \eqref{Fhm} corresponds to a solution of one of the equations in \eqref{Ghm}, for $j \in \{1, \ldots, 3^k\}$.  
  
The extra assumption on the factorization of $F$ modulo prime factors of~$h$ ensures that the equations
$F(x , y) = 0$
and 
$G_{j}(x , y) = 0$
are solvable over $\mathbb{Z}_{p}$ for all primes $p\mid h$ as well. 
Indeed, given $p\mid h$, if $F(x , y) \equiv L_1(x,y) L_2(x , y)^{2}\, \, \textrm{ (mod}\, \,  p)$, then the curve $F(x,y)=h$ over $\F_p$ has a smooth point (namely, on the component $L_1(x,y)=0$); this smooth point will then lift to a $\Z_p$-point on $F(x,y)=h$. Thus the equations $F(x,y)=h$ and $G_j(x,y)=h$ will again be everywhere locally soluble.

    The rest of the proof is identical to the above proof for $h=1$.

\section{Proof of Theorem ~\ref{maingeneral}}\label{proofofmaingeneral}

 We turn next to Thue equations of general degree.  Since asymptotics for the number $\GL_2(\Z)$-classes of binary $n$-ic forms of bounded discriminant are not known for $n>3$ (see~\cite{BM}), in this section we instead order all integral binary $n$-ic forms by the maximum of the absolute values of their coefficients, but otherwise follow a strategy analogous to that of Section~\ref{proofofmain}.
 
 We again begin by restricting first to the case $h =1$. 
 
 \subsection{Construction of binary $n$-ic forms $F(x,y)$ satisfying certain congruence conditions and having large discriminant}\label{fconstruction2}
 
 Let $p_{i}$ be the $i$-th prime greater than or equal to  $n$ and set
 $$
 m = \prod_{i=1}^k p_{i},
 $$ 
 where $k$ is a positive integer to be chosen later. 
  Let $F(x , y)$ be a  maximal primitive irreducible integral binary form  of degree $n$ with Galois group $S_n$ such that
 $$
\left|\textrm{disc}(F)   \right| > (3.5)^{2n(n-1)} \,  n^{2n} \left( \prod_{i=1}^{k} p_{i} \right)^{4(n-1)}.
 $$
 
We assume that the leading coefficient of $F$ is positive, so that $F$ takes positive values.  Assume further that $F(x , y)$  splits completely modulo the primes $p_{1}$, \ldots, $p_{k}$.
Let $g=(n-1)(n-2)/2$ be the genus of the curve $z^n = F(x , y)$, and  assume that for every prime $p \not \in \{ p_{1}, \ldots , p_{k}\}$, with $ p < (2g + 1)^2$, we have
$$
F(x , y) \equiv L_1(x,y) L_2(x , y)^{n-1}\, \, \textrm{ (mod}\, \,  p^{1+v_p(n)+\delta_2(p,n)}),
$$
where $L_1(x,y)$ and $L_2(x,y)$ are linear forms that are linearly independent modulo~$p$, $v_p(n)$ denotes the $p$-adic valuation of $n$, and $\delta_2(p,n)$ is equal to 1 if $p=2$ and $p\mid n$, and is equal to 0 otherwise. 

Finally, assume, for each prime $p> (2g+1)^2$,
that $F(x,y)$ does not factor as $cM(x,y)^r$ modulo~$p$ for any integer  $r>1$ and any binary form $M(x,y)$ and constant $c$ over $\F_p$.
We will show in Subsection~\ref{posdens2} that such forms $F$ 
give a positive proportion of all classes of binary $n$-ic forms when ordered by height.

 \subsection{Construction of forms $G_j$ associated to each $F$}\label{gjconstruction2}

 Assume that
\begin{equation}\label{abcsplitg}
 F(x , y) \equiv m_{0} (x - a_{1}y)(x-a_{2}y) \ldots (x- a_{n}y)\, \, (\textrm{mod} \, \,   p_{1}),
 \end{equation}
 where $m_{0} \not \equiv 0$ (mod $p_{1}$), and $a_{1}$, \ldots, $a_{n}$ are distinct integers modulo $p_{1}$. 
 Similarly to the cubic case in Section \ref{proofofmain}, and following the method of Bombieri and Schmidt \cite{Bom}, for $i = 1, \ldots, n$, define
 \begin{eqnarray*}
 F_{a_{i}}(x , y) & : = & F(p_{1} x + a_{i} y, y).
 \end{eqnarray*}
The condition that $F$ has Galois group $S_n$ over $\Q$ again implies that $F$ must have trivial stabilizer in $\GL_2(\Q)$; thus, by the same argument as in the binary cubic case, the $n$ forms $F_{a_i}(x,y)$ are again seen to be pairwise inequivalent. 

By Lemma \ref{content}, the content of $F_{a_{i}}(x , y)$ is exactly divisible by $p_{1}$ for every $i=1, \ldots, n$.  Thus we may divide all the coefficients of the $F_{a_{i}}$'s by $p_{1}$ to obtain $n$ pairwise inequivalent binary $n$-ic forms $\tilde{F}_{a_{i}}$, for $i=1, \ldots, n$, with integral coefficients and content $1$. 
 By Lemma \ref{L7Bom}, each solution of the equation $F(x , y) = m$ corresponds uniquely to a solution of one of the equations
 $$
 \tilde{F}_{a_{i}}(x,y) = m/p_{1}, \, \, i=1, \ldots, n.
 $$
 
  Since we assumed $F(x , y)$ splits completely modulo $p_{1}$, we either have \eqref{abcsplitg} or
$$
F(x , y) \equiv m_{0} y (x - a_{2}y) \ldots  (x - a_{n}y) \, \, (\textrm{mod}\, \,  p_{1}),
$$
for some integers  $m_{0} \not \equiv 0$  (mod $p_{1}$),  and distinct $a_{2}, \ldots, a_{n}$ modulo $p_{1}$. In the latter case, we define $F_{a_{2}}$, \ldots, $F_{a_{n}}$ as above and we take $a_{1} = \infty$ and define
$$
F_{\infty}(x, y) := F (p_{1}y , x)
$$
and 
$$
\tilde{F}_{\infty}(x, y) := \frac{1}{p_{1}}\, F_{\infty}(x, y). 
$$
Just as in the binary cubic case, we see that $F_{\infty}(x, y)$ is not $\GL_2(\Z)$-equivalent to any of $F_{a_{i}}(x, y)$, for $i = 2, \ldots, n$. 

In the statement of Theorem ~\ref{maingeneral}, we order integral binary $n$-ic forms  by their naive heights, where the ({\it naive}) {\it height} of an integral binary $n$-ic form is defined as the maximum of the absolute values of all its coefficients. We now proceed to show that the height of each of the $\tilde{F}_{a_{i}}$'s can be bounded by a function of the height of $F$. Let 
$$
F(x , y) = f_{0}x^n + f_{1} x^{n-1} y + \ldots + f_{n}y^n \in \mathbb{Z}[x , y]
$$ 
be a binary form of degree $n$. Suppose that 
\begin{equation}\label{alessp}
F (a , 1) \equiv 0 \, \, \, (\textrm{mod}\, \, p), \,  \, \, \textrm{with}\, \, a \in \mathbb{Z}.
\end{equation}
 Put 
\begin{equation}\label{Faei}
F_{a}(x , y) =    F(p x + ay , y) =  e_{0}x^n + e_{1} x^{n-1} y + \ldots + e_{n}y^n.
\end{equation}
Then
\begin{equation}\label{dotss}
 e_{n-j} = p^j \sum_{i=0}^{n-j} f_{i} a^{n-i-j} {n-i \choose j}.
 \end{equation}
  If $j \geq 1$, clearly $e_{n-j}$ is divisible by $p$.  Since  $e_{n} = F(a , 1)$, by \eqref{alessp}, $e_{n}$ is also divisible by $p$, as is expected by Lemma ~\ref{content}. We also notice that  $e_{n-1} = p f'(a)$, where $f'(X)$ denotes the derivative of polynomial $f(X) = F(X , 1)$. Therefore, if $a$ is a simple root modulo $p$  then $f'(a)\not \equiv 0\, \, (\textrm{mod}\, p)$ and 
$$
\tilde{F_{a}}(x , y) = \frac{F_{a}(x , y)}{p}
$$
is congruent to 
\begin{equation}\label{yL}
y^{n-1} L(x , y)\, \,  (\textrm{mod}\, \, p),
\end{equation}
 where $L(x , y)= l_{1}x + l_{2}y$  is a linear  form modulo $p$, with $l_{1} \not \equiv 0\pmod p$.  Lemma \ref{newlemmaonlocalsol}, which will be proven in Subsection \ref{locsol2},  implies
 that $\tilde{F_{a}}(x , y) \equiv 1 \, \, (\textrm{mod}\, \, p)$ is solvable. 
 Further, we may choose the  integer $a$ so that 
\begin{equation}\label{pap}
-\frac{p-1}{2} \leq a \leq \frac{p-1}{2}.
\end{equation}
Let $H(F)$ be the naive  height of $F$. Then, by \eqref{dotss}, \eqref{Faei}, and \eqref{pap}, we have
\begin{eqnarray*}
\left|e_{n-j} \right| &\leq& p^{j} H(F) \left( \frac{p-1}{2}\right)^{n-j} \sum_{i=0}^{n-j} \left( \frac{p-1}{2}\right)^{-i} {n-i \choose j}\\
&\leq& p^{j} H(F) \left( \frac{p-1}{2}\right)^{n-j} \sum_{i=0}^{n-j}  {n-i \choose j}\\
& =  &   p^{j} H(F) \left( \frac{p-1}{2}\right)^{n-j} {n+1 \choose j+1}\\
& \leq &   p^{n} H(F){n+1 \choose j+1}.
\end{eqnarray*}
By Stirling's approximation, for every $j$, we have 
$$
{n+1 \choose j+1} \leq \frac{e\,  2^{n+2}}{2 \pi \sqrt{n+1}} < \frac{  2^{n+1}}{\sqrt{n+1}}.
$$
Therefore, we conclude that 
\begin{equation*}
H(F_{a}) < p^{n} \frac{  2^{n+1}}{\sqrt{n+1}} \,  H(F),
\end{equation*}
and consequently
\begin{equation}\label{Htilde}
H(\tilde{F_{a}}) < p^{n-1} \frac{  2^{n+1}}{\sqrt{n+1}} \,  H(F).
\end{equation}
Clearly,
$$
H(\tilde{F}_{\infty}) \leq p^{n-1} H(F),
$$
by definition.

Applying the reduction procedure for each prime $p_{i}$,  $i=1, \ldots , k$, we obtain $n^k$ mutually inequivalent binary forms of degree $n$ with integral coefficients, content $1$, and naive height bounded above by
$$
\left(\prod_{i=1}^{k} p_{i}\right)^{n-1}  \frac{2^{(n+1)k}}{(n+1)^{k/2}}\,   H(F).
$$
We denote these binary forms by $G_{j}$ for $j=1, \ldots, n^k$. 

If $F(x , y)$ is  irreducible over $\mathbb{Q}$, then  $G_{j} (x , y)$ will be irreducible over $\mathbb{Q}$ as well.
 Finally, the $G_{j}$'s are not constructed as proper subforms of  $F(x , y)$, and so the $G_j$'s must be maximal forms as well.  Indeed, they are maximal over $\Z_p$ for all $p\notin \{p_1,\ldots,p_k\}$ (being equivalent, up to a unit constant, to $F(x,y)$ over $\Z_p$ in that case), while for $p\in\{p_1,\ldots,p_k\}$, we have $p\nmid \disc(F)$, implying $p^{(n-1)(n-2)}||\disc(G_j)$, so $G_j$ cannot be a subform over $\Z_p$ of any form by equation (\ref{St6}). Hence the $G_j$'s are all maximal forms.  
 
  \subsection{Many of the $G_j$'s do not represent 1}\label{gjno12}
  
   By Theorem~\ref{maineq}, because of the condition on the discriminant of $F$, we conclude that the equation 
 $$
 F(x , y) = m
 $$
 has at most  $13 n$ primitive solutions if $n =3, 4$, and has at most  $11n $ primitive solutions if $n \geq 5$ (take $\epsilon = \frac{1}{4(n-1)}$).  
 
Each primitive solution $(x_{0}, y_{0})$ to the equation
$$
F(x , y) = m = \prod_{i=1}^{k} p_{i}
$$
corresponds to a unique triple  $(j, x_j, y_j)$ so that $(x_j, y_j)$ is a solution to the equation 
$$
G_{j}(x , y) = 1.
$$
Since the equation $F(x , y) = m$ has at most $13n  $ primitive solutions,   at least $n^{k} -  13n$ binary forms among the $G_{j}$'s cannot represent $1$, if $n = 3, 4$. Similarly, if $n \geq 5$, at least $n^{k} -  11n$ forms among the $G_{j}$'s cannot represent $1$.
  Let $n$ be even and assume that a solution $(x_{0}, y_{0})$ to $F(x , y) = m$ corresponds to a solution $(\tilde{x}_{0}, \tilde{y}_{0})$ of  $G_{I} = 1$ for a unique $I \in \{1, \ldots, n^k\}$. Clearly  the solution $(-x_{0}, -y_{0})$ to $F(x , y) = m$  will correspond to the  solution $(-\tilde{x}_{0}, -\tilde{y}_{0})$ of  $G_{I} = 1$. Hence, as in Theorem~\ref{maineq}, we may consider $(x_{0}, y_{0})$ and $(-x_{0}, -y_{0})$ as one solution.

  Thus starting from any maximal irreducible binary  form $F(x,y)$ of degree $n = 3$ or $4$ that splits completely modulo $p_1,p_2,\ldots,p_k$, we have produced at least $n^k - 13n$ mutually inequivalent maximal irreducible binary  forms $G(x,y)$  of degree $n$ that do not represent 1. 
 In order to produce a positive number of such forms $G(x,y)$, it suffices to take 
 $$
 k > \frac{\log(13n) }{\log n}.
 $$
 Similarly, for $n \geq 5$, we may choose 
 $$
 k > \frac{\log(11 n) }{\log n}
 $$
 to  produce at least $n^k - 11n$ maximal irreducible binary  forms $G(x,y)$  of degree $n$ that do not represent 1. 

 
  \subsection{The forms $G_j$ represent  $1$ everywhere locally}\label{locsol2}
  

Next we wish to show that these binary forms $G_{j}(x , y)$ that do not globally represent $1$ do represent $1$ everywhere locally.   

We first use the following simple lemma to show that the forms $G_{j}$'s represent $1$ modulo every prime 
$p< (2g + 1)^2$. 

\begin{lemma}\label{newlemmaonlocalsol}
Let $p$ be a prime and assume that $F(x , y)$ is a primitive binary $n$-ic form  such that  
$F(x , y) \equiv L_1(x,y) L_2(x , y)^{n-1} (\textrm{mod}\, \,  p^{1+v_p(n)+\delta_2(p,n)}),$
where $L_1(x,y)$ and $L_2(x,y)$ are linearly independent linear forms modulo $p$.
Then $F$ represents $1$ over~$\mathbb{Z}_p$.
\end{lemma}
\begin{proof}
We may find $(x_0,y_0)$ such that 
$$
L_1(x_0,y_0)\equiv L_2(x_0,y_0)\equiv 1 \, (\textrm{mod}\,  p^{1+v_p(n)+\delta_2(p,n)}); $$
indeed, since $L_1$ and $L_2$ are linearly independent forms modulo $p$, such $x_0,y_0\in\Z/p^{1+v_p(n)+\delta_2(p,n)}\Z$ exist. 
It follows that $F$ represents 1 modulo $p^{1+v_p(n)+\delta_2(p,n)}$.  Since a unit in $\Z_p$ is an $n$-th power if and only if it is an $n$-th power modulo $p^l$ for sufficiently large $l$ - in fact, $l=1+v_p(n)+\delta_2(p,n)$ suffices - we see that $F$ represents unit $n$-th powers and therefore 1 over $\Z_p$, as desired.
\end{proof}
For any prime number $p< (2g+1)^2$, we see from Lemma~\ref{newlemmaonlocalsol} that each $G_j$ represents 1 in $\Z_p$.  

By the Hasse--Weil bound,  the number of points $N$ on a curve $C$ of genus~$g$ over the finite field $\mathbb{F}_{q}$ of order $q$ satisfies the inequality
$$
\left| N - (q+1) \right| \leq 2g \sqrt{q}.
$$
For $p> (2g+1)^2$, we assumed that the forms $F$, and thus the forms $G_j$, 
do not factor as $c\, M(x,y)^r$ modulo~$p$ for any $r >1$ and any binary form $M(x,y)$ and constant $c$ over $\F_p$; it follows from this assumption that the curves $z^n=F(x,y)$ are irreducible over $\F_p$. 
The Hasse-Weil bound thus applies to each such curve $z^n=F(x,y)$ to produce a smooth point. Therefore by Hensel's lemma,  such a smooth $\mathbb{F}_p$-point will lift to a $\mathbb{Z}_p$-point. We conclude that the forms $G_{j}$ that we have constructed all represent $1$ in $\Z_p$ for primes $p> (2g+1)^2$ as well. 

Finally, the $G_j$'s continue to represent positive values and therefore also represent 1 over $\R$. In summary, we have proven the following theorem.
 
\begin{thm}\label{mainconditiongeneral}
Let $n \geq 3$  be an integer and let $p_{i}$ be the $i$-th  prime greater than or equal to~$n$. Set 
 $m = \prod_{i=1}^k p_{i}$ where $k$ is any integer satisfying
  \[ \left\{
  \begin{array}{l l}
   k>  \frac{\log(13n ) }{\log n} & \quad \text{if} \, \, n = 3, 4\\
  k >  \frac{\log(11n ) }{\log n} & \quad \text{if}\, \,  n \geq 5.
  \end{array} 
 \right.  \]
  Let $F(x , y)$ be a primitive maximal irreducible integral binary $n$-ic form with~Galois group $S_n$ and discriminant greater than 
 $(3.5)^{2n(n-1)} n^{2n} m^{4(n-1)}$.
 Assume further that $F(x , y)$ has positive first coefficient and splits completely modulo the primes $p_{1}, \ldots, p_{k}$.
 Let $g=(n-1)(n-2)/2$ be the genus of the curve $z^n = F(x , y)$ and  assume that for every prime $p$, with  $p < p_{1}$ or $p_{k}< p <  (2g+1)^2$, we have
$$
F(x , y) \equiv L_1(x,y) L_2(x , y)^{n-1}\, \, (\textrm{mod}\, \,  p^{1+v_p(n)+\delta_2(p,n)}),
$$
where $L_1(x , y)$ and $L_2(x,y)$ are linearly independent linear  forms modulo~$p$. 
Moreover, assume that for any prime $p >  (2g+1)^2$, 
the form $F(x,y)$ does not factor as $c\, M(x,y)^r$ modulo~$p$ for any $r>1$ and any binary form $M(x,y)$ and constant $c$ over $\F_p$.

Given any such form $F$, there exist $n^k$ mutually inequivalent irreducible maximal binary $n$-ic forms $G_{j}$, for $j =1, \ldots, n^k$, such that  at least
  \[ \left\{
  \begin{array}{l l}
   n^k-13 n & \quad \text{if} \, \, n = 3, 4\\
  n^k-11 n & \quad \text{if}\, \,  n \geq 5
  \end{array} 
 \right.  \]
 of the $G_j$'s satisfy the following:
\begin{enumerate}[{\rm (i)}]
\item they locally everywhere represent $1;$
\item  they globally do not represent $1;$ and 
\item  they are maximal forms.
\end{enumerate}
\end{thm}

 \subsection{The forms $F$ and also the forms $G_j$ yield a positive proportion of all integral binary $n$-ic forms}\label{posdens2}

We now verify that the forms~$F$ satisfying the conditions of 
Theorem~\ref{mainconditiongeneral} give a positive proportion of all integral binary forms of degree $n$ when ordered by naive height.

\begin{thm}\label{dav2} 
Let $S$ be a set of integral binary $n$-ic forms defined by congruence conditions modulo bounded powers of primes $p$ where, for sufficiently large $p$, the defining congruence conditions at $p$ exclude only a set that modulo $p$ is contained in the set of $\F_p$-points of a fixed codimension $\geq 2$ subscheme of~$\mathbb A^{n+1}_\Z$.
Let $N(S;X)$ denote the number of irreducible binary $n$-ic forms with Galois group $S_n$ that are contained in $S$ and have positive leading coefficient and naive height less than~$X$.  Then:
\begin{enumerate}[{\rm (a)}]
\item $N(S;X) \sim 2^nX^{n+1} \prod_p\mu_p(S);$
\item  $N(S;X) \sim 2^nX^{n+1}\prod_p\mu_p(S),$
\end{enumerate}
where $\mu_p(S)$ denotes the $p$-adic density of $S$ in the space of integral binary~$n$-ic forms. 
\end{thm}
\begin{proof}
Note that $100\%$ of integral binary~$n$-ic forms, when ordered by naive height, are irreducible and have Galois group $S_n$ by Hilbert's irreducibility theorem. The theorem is then obtained as an immediate consequence of Ekedahl's sieve (see, e.g., \cite{Eke}, \cite{PS}, or \cite{GeoSieve} for expositions). 
\end{proof}

Theorem~\ref{dav2}
 states that the number of irreducible binary $n$-ic forms having Galois group $S_n$ and naive height less than~$X$ grows as a constant times~$X^{n+1}$, even if we are counting only those forms satisfying a finite, or even a suitable infinite, set of congruence conditions.  
 In particular, Theorem~\ref{dav2} immediately implies that the set $S_0$ of all $F$ satisfying the conditions of Theorem~\ref{mainconditiongeneral} gives a positive proportion of all integral binary $n$-ic forms, when these forms are ordered by naive height. 
 
 Indeed,  we may determine the exact density of such $F$ among all irreducible binary $n$-ic forms as follows. 
Again, if a form is non-maximal, then either it is not primitive, or after an $\SL_2(\Z)$-transformation it is of the form $a_0x^n+a_1x^{n-1}y+\cdots+a_ny^n$, where $p^i\mid a_i$, $i=1,\ldots,n$, for some prime~$p$.  In particular, integral binary $n$-ic forms that are non-maximal must factor modulo some prime $p$ as a constant times the $n$-th power of a linear form.  Thus for primes $p> (2g +1)^2$, the $p$-adic density of binary $n$-ic forms  that are primitive and not constant multiples 
of  $r$-th powers of binary forms modulo~$p$ for any $r>1$ (and thus are automatically maximal over $\Z_p$) is
\begin{equation}\label{mup}
\Upsilon_p(S_0)=\sum_{r\mid n}\mu(r)\frac{p^{n/r+1}-1}{p^{n+1}},
\end{equation}
by inclusion--exclusion; here the $r$-th term corresponds to the density of binary $n$-ic forms that are constant multiples modulo $p$ of exact $r$-th powers. 

The~$p_{i}$-adic density (for $i=1, \ldots, k$) of all primitive forms that split completely modulo $p_{i}$ (and thus are automatically maximal over $\Z_{p_i}$) is
$$
\frac{{p_i+1 \choose n} (p_i-1)   }{p_{i}^{n+1}}.
$$

The $p$-adic density of  primitive binary $n$-ic forms of the shape
\begin{equation}\label{modhigherp}
 L_1(x,y) L_2(x , y)^{n-1}\, \, (\textrm{mod} \, \,  p^l)
\end{equation}
where $L_1(x,y)$ and $L_2(x,y)$ are linearly independent linear forms modulo $p$ and 
$l=1+v_p(n)+\delta_2(p,n)$ (such forms are again automatically maximal over $\Z_p$) is 
$$
\frac{(p+1)p(p-1)p^{3(l-1)}}{p^{l(n+1)}}.
$$

Therefore, the density of $F$ satisfying the conditions of Theorem~\ref{mainconditiongeneral} among all classes of irreducible integral binary  forms of degree $n$, when these forms are ordered by height, is at least
$$\frac12\cdot
\prod_{i=1}^{k} \frac{{p_{i}+1 \choose n} (p_{i}-1)   }{p_{i}^{n+1}} \prod_{ \substack{p\notin\{p_1,\ldots,p_k\}
\\ p < (2g + 1)^2} }\frac{(p^3-p)p^{3(l-1)}}{p^{l(n+1)}} \prod_{p >(2g + 1)^2} \Upsilon_p,
 $$
 where $l=1+v_p(n)+\delta_2(p,n)$ depends on $p$, $\Upsilon_p$ is given as in \eqref{mup}, and  $g=(n-1)(n-2)/2$ is the genus of the curve $z^n = F(x , y)$.

 In order to show that the mutually inequivalent forms $G_{j}$ that we constructed from these forms~$F$, which are maximal, locally represent $1$, but do not globally represent 1, yield a positive proportion of all integral binary $n$-ic forms when ordered by height, we note that each $F$ gives rise to $n^k - 13n$ forms $G_{j}$ (if $n =3, 4$), each satisfying 
 $$
 H(G_{j}) < \left(\prod_{i=1}^{k} p_{i}\right)^{n-1}  \frac{2^{(n+1)k}}{(n+1)^{k/2}}\,   H(F),
 $$ 
by \eqref{Htilde}.
We conclude that the density of our forms $G_{j}$, when ordered by height, among all integral binary $n$-ic forms is at least 
$$
(n^k -13n) \frac{(n+1)^{k/2}}{2^{(n+1)k+1}}
 \prod_{i=1}^{k} \frac{{p_{i}+1 \choose n} (p_{i}-1)   }{p_{i}^{2n}} \prod_{ \substack{p\notin\{p_1,\ldots,p_k\}
\\ p < (2g + 1)^2} }\frac{(p^3-p)p^{3(l-1)}}{p^{l(n+1)}} \prod_{p > (2g+1)^2} \Upsilon_p,
 $$
   where again $l=1+v_p(n)+\delta_2(p,n)$ depends on $p$, $\Upsilon_p$ is given in \eqref{mup} and  $g=(n-1)(n-2)/2$ is the genus of the curve $z^n = F(x , y)$, and $n =3, 4$.

Similarly, if $n\geq 5$, we conclude that the density of our forms $G_{j}$, when ordered by height, among all integral binary $n$-ic forms is at least 
$$
(n^k -11n) \frac{(n+1)^{k/2}}{2^{(n+1)k+1}}
\prod_{i=1}^{k} \frac{{p_{i}+1 \choose n} (p_{i}-1)   }{p_{i}^{2n}} \prod_{ \substack{p\notin\{p_1,\ldots,p_k\}
\\ p < (2g + 1)^2} }\frac{(p^3-p)p^{3(l-1)}}{p^{l(n+1)}}\prod_{p > (2g+1)^2} \Upsilon_p.
 $$

\subsection{The case of general $h$}\label{genh2}

We may handle general $h\in\Z$ in a similar manner. 
Let $k$ be a positive integer 
 satisfying
  \[ \left\{
  \begin{array}{l l}
  k>  \frac{\log(13n ) }{\log n} & \quad \text{if} \, \, n = 3, 4\\
   k >  \frac{\log(11n ) }{\log n} & \quad \text{if}\, \,  n \geq 5.
  \end{array} 
 \right.  \]
Let $p_{1}, \ldots, p_{k}$ be the $k$ smallest  primes that are greater than or equal to $n$ such that $p_{i} \nmid h$, and 
set  $m = \prod_{i=1}^{k} p_{i}$.

 Let $F(x , y)$ be an irreducible n-ic binary form,
 with 
 $$
 |\textrm{disc}(F)| > (3.5)^{2n(n-1)} n^{2n} \left( h m \right)^{4(n-1)}.
 $$
  Then by Theorem \ref{maineq}, the equation  
  \begin{equation}\label{Fhm1}
  F(x , y)= h m
  \end{equation}
  has at most $13n$  primitive solutions if $n = 3, 4$, and has at most $11 n$ primitive solutions if $n \geq 5$. 
  
  In addition to the assumptions on the form $F$ that were made in Subsection~\ref{fconstruction2}, we also assume    that for every prime $p \mid h$, we have 
  $$
F(x , y) \equiv L_1(x,y) L_2(x , y)^{n-1}\, \, \textrm{ (mod}\, \,  p),
$$
where $L_1(x,y)$ and $L_2(x,y)$ are linear forms over $\F_p$ that are linearly independent. 

  Similarly to the case $h=1$, performing the construction in Lemmas~\ref{L7Bom} and \ref{content} as in Subsection~\ref{gjconstruction2} using the primes $p_1,\ldots,p_k$, we obtain  $n^k$ mutually inequivalent degree-$n$ Thue equations of the shape
   \begin{equation}\label{Ghm1}
  G_{j}(x , y) = h,
  \end{equation}
  so that every solution $(x_{0}, y_{0})$ to the equation \eqref{Fhm1} corresponds to a solution of one of the equations in \eqref{Ghm1}, for $j \in \{1, \ldots, n^k\}$.  
  
The extra assumption on the factorization of $F$ modulo prime factors of~$h$ again ensures, as in Subsection~\ref{genh} in the case $n=3$, that the equations
$F(x , y) = 0$
and 
$G_{j}(x , y) = 0$
are solvable over $\mathbb{Z}_{p}$ for all primes $p\mid h$ as well. 


 The rest of the proof is identical to the above proof for $h=1$.

 \section*{Acknowledgments}
  
    We are grateful to the Mathematisches Forschungsinstitut Oberwolfach, where this work was initiated in July  2011.
  We also thank T.\ Browning, J.-L.\ Colliot-Th\'el\`ene, K.\ Gy\H{o}ry, D.\ Schindler, A.\ Shankar, X.\ Wang, and the anonymous referees for helpful conversations and for comments on an earlier draft of this manuscript.  The first author was partially supported by NSF grant  DMS-1601837. 
  The second author was partially supported by a Simons Investigator Grant and NSF Grant DMS-1001828.
 



\begin{thebibliography}{99}


\bibitem{AkhQuaterly}  {\sc S.~Akhtari}, \textit{Representation of small integers by binary forms}, Q. J. Math $66$ (4)  (2015), 1009-1054.



\bibitem{hyperE}  {\sc  M.~Bhargava}, \textit{Most hyperelliptic curves over $\mathbb{Q}$ have no rational points}, 
{\tt arXiv:1308.0395v1}, preprint.


\bibitem{GeoSieve}  {\sc  M.~Bhargava}, \textit{The geometric sieve and the density of squarefree values of invariant polynomials}, 	{\tt arXiv:1402.0031}, preprint.



\bibitem{BG} {\sc M.~Bhargava and B.~Gross}, \textit {The average size of the
  $2$-Selmer group of Jacobians of hyperelliptic curves having a
  rational Weierstrass point}, Automorphic representations and L-functions, 23-91, Tata Inst. Fundam. Res. Stud. Math., $22$, Tata Inst. Fund. Res., Mumbai (2013).

\bibitem{BST}  {\sc M.~Bhargava, A.~Shankar, and  J.~Tsimerman}, \textit{On the Davenport-Heilbronn theorem and second order terms}, Invent. Math.  $193$ (2013), no. $2$, 439-499.

\bibitem{BY}  {\sc  M.~Bhargava and A.~Yang}, \textit{On the number of integral binary $n$-ic forms having bounded Julia invariant}, {\tt https://arxiv.org/abs/1312.7339}. 

\bibitem{BM} {\sc B. J.~Birch and J. R.~Merriman}, \textit{Finiteness 
theorems for binary forms with given discriminant}, Proc. London Math. Soc. $25$ (1972), 385-394.


\bibitem{Bom}  {\sc E.~Bombieri and  W. M. Schmidt}, \textit{On Thue's equation}, Invent. Math. $88$ (1987), 69-81.



\bibitem{CTX}  {\sc J.-L.~Colliot-Th\'el\`ene and F.~Xu}, {\textit 
Brauer-Manin obstruction for integral points of homogeneous spaces and representation by integral quadratic forms}, Compositio Math. $145$ (2009), 309-363.


%


\bibitem{Dav14I} {\sc H.~Davenport}, \textit{On the class-number of binary cubic forms I}, J. London Math. Soc. $26$ (1951), 183-192.

\bibitem{Dav14II} {\sc H.~Davenport}, \textit{On the class-number of binary cubic forms II}, J. London Math. Soc. $26$ (1951), 192-198.

\bibitem{DH}  {\sc H.~Davenport and H.~Heilbronn}, \textit{On the density of discriminants of cubic fields II}, Proc. Roy. Soc. London Ser. A $322$ (1971), no. 1551, 405-420.

\bibitem{DM} {\sc R.~Dietmann and  O.~Marmon}, \textit{Random Thue and Fermat equations},
Acta Arith. $167$ (2015), 189-200.


\bibitem{Eke} {\sc T.~Ekedahl}, \textit{An infinite version of the Chinese remainder theorem}, Comment. Math. Univ.
St. Paul. $40$ (1991), no. 1, 53-59.



\bibitem{EG} {\sc J. H.~Evertse and K.~Gy\H{o}ry}, \textit{Thue 
inequalities with a small number of solutions}, in: The mathematical heritage of C.F. Gauss, World Scientific Publ. Co., Singapore, 1991, 204-224.






\bibitem{Heu} {\sc C.~Heuberger}, \textit{Parametrized Thue equations -A survey}, Proceedings of the RIMS symposium ``Analytic Number Theory and Surrounding Areas'', held in Oct 2004, RIMS, Kyoto, vol. 1511, (2006) 82-91.



\bibitem{JS} {\sc J.~Jahnel and D.~Schindler}, {\textit  On the number of certain Del Pezzo surfaces of degree four violating
the Hasse principle}, J. Number Theory $162$ (2016), 224-254.

\bibitem{Mit} {\sc V.~Mitankin},  \textit{Failures of the integral Hasse principle for affine quadric surfaces}
J. London Math. Soc.
$95$  (2017) 1035-1052.


\bibitem{PS} {\sc B.~Poonen and M.~Stoll}, \textit{A local-global principle for densities}, in: Scott D. Ahlgren (ed.) et al.: Topics in number theory. In honor of B. Gordon and S. Chowla. Kluwer Academic Publishers, Dordrecht. Math. Appl., Dordr. 467 (1999), 241-244.

\bibitem{PS2}
{\sc B.\ Poonen and M.\ Stoll},
\textit{Most odd degree hyperelliptic curves have only one rational point}, Ann. of Math. $180:3$ (2014) 1137-1166.

\bibitem{shintani} {\sc T.~Shintani}, \textit{On Dirichlet series whose coefficients are class numbers of integral binary cubic forms}, J. Math. Soc. Japan $24$, no. 1 (1972), 132-188.

\bibitem{Ste} {\sc C. L.~Stewart},  \textit{On the number of solutions of polynomial congruences and Thue equations}, J. Amer. Math. Soc. $4$ (1991), 793-838.



\bibitem{Tho} {\sc E.~Thomas},  \textit{Complete solutions to a family of cubic diophantine equations}, J. Number Theory $34$ (1990), 235-250.

\bibitem{Thue}
{\sc Thue}, Berechnung aller L\"osungen gewisser Gleichungen von der Form, Vid. Skrifter I Mat.-Naturv. Klasse (1918), 1--9.



\end{thebibliography}
\end{document}